\documentclass[11pt,a4paper,twoside]{amsart}

\usepackage[english]{babel}
\usepackage[T1]{fontenc}
\usepackage{microtype, fancyhdr, lmodern, xcolor}
\usepackage[a4paper, hmarginratio=1:1]{geometry}

\usepackage{amsfonts, amsmath, amsthm, amssymb, mathrsfs, amscd}		
\numberwithin{equation}{section}	
\allowdisplaybreaks
\usepackage{breqn}
\usepackage{faktor}
\allowdisplaybreaks

\usepackage{url, enumitem}
\usepackage[noadjust]{cite}

\usepackage{graphicx}
\usepackage{multirow,bigdelim}
\usepackage{mathdots}

\usepackage{hyperref}

\theoremstyle{plain}
\newtheorem{thm}{Theorem}[section]
\newtheorem{propo}[thm]{Proposition}

\theoremstyle{definition}
\newtheorem{de}[thm]{Definition}
\newtheorem{ex}[thm]{Example}

\theoremstyle{remark}
\newtheorem{rmk}[thm]{Remark}

\renewcommand{\epsilon}{\varepsilon}

\newcommand{\largewedge}{\mbox{\large $\wedge$}} 	

\newcommand{\Rr}{\mathbb{R}}
\newcommand{\Zz}{\mathbb{Z}}

\title{Alexander invariants of ribbon tangles and planar algebras}

\author{Celeste Damiani}
\address{%
Laboratoire de Math\'{e}matiques Nicolas Oresme (LMNO) CNRS UMR 6139 \\
Universit\'{e} de Caen BP 5186\\
14032 CAEN Cedex France
}
\email{celeste.damiani@unicaen.fr }

\author{Vincent Florens}
\address{%
Laboratoire de Math\'{e}matiques et leurs applications, UMR CNRS 5142\\
Universit\'{e} de Pau et des Pays de l'Adour\\
Avenue de l'Universit\'{e}\\
 BP 1155 64013 PAU Cedex France
}
\email{vincent.florens@univ-pau.fr}

\begin{document}

\begin{abstract}
Ribbon tangles are proper embeddings of tori and cylinders in the $4$-ball~$B^4$, ``bounding''~$3$-manifolds with only ribbon disks as singularities. We construct an Alexander invariant~$\mathsf{A}$ of ribbon tangles equipped with a representation of the fundamental group of their exterior in a free abelian group $G$. This invariant induces a functor in a certain category $\mathsf{R}ib_G$ of tangles, which restricts to the exterior powers of Burau-Gassner representation for ribbon braids, that are analogous to usual braids in this context. We define a circuit algebra $\mathsf{C}ob_G$ over the operad of smooth cobordisms, inspired by diagrammatic   planar   algebras introduced by Jones~\cite{J},   and prove that   the invariant $\mathsf{A}$ commutes with the compositions in this algebra. On the other hand,
 ribbon tangles admit diagrammatic representations, throught welded diagrams. We give a simple combinatorial description of $\mathsf{A}$ and of the algebra $\mathsf{C}ob_G$, and observe that our construction is a topological incarnation of the Alexander invariant of Archibald~\cite{Arc}.
 When restricted to diagrams without virtual crossings, $\mathsf{A}$ provides a purely local description of the usual Alexander poynomial of links, and extends the construction by Bigelow, Cattabriga and the second author~\cite{BCF}.
\end{abstract}

\maketitle

\section{Introduction}
A \emph{ribbon torus link} is a locally flat embedding of disjoint tori $S^1 \times S^1$ in the ball~$B^4$, bounding locally flat immersed solid tori $S^1 \times D^2$ whose singular sets are finite numbers of \emph{ribbon disks}. These elementary singularities are $4$-dimensional analogues of the classical notion of ribbon introduced by Fox~\cite{F}.
Ribbon torus links are considered up to ambient isotopy; the fundamental groups of their complement and  the derived Alexander modules provide topological invariants.

In this paper, we consider~\emph{ribbon tangles} that are ribbon proper embeddings of disjoint tori and cylinders in the ball~$B^4$. The intersection of a given tangle~$T$ with~$\partial B^4=S^3$ is a trivial link~$L$. For a given free abelian group~$G$, tangles are~\emph{colored} with a group homomorphism~$\varphi\colon H_1(B^4 \setminus T) \to G$.
We construct an \emph{Alexander invariant}~$\mathsf{A}$ of colored ribbon tangles, lying in the exterior algebra of
 the homology~$\mathbb{Z}[G]$-module~$H_1^\varphi(S^3 \setminus L)$, twisted by the morphism induced by~$\varphi$, see Definition~\ref{alex}. This invariant coincides with the Alexander polynomial in the case of tangles with only two boundary components, see Proposition~\ref{oneone}. The construction of~$\mathsf{A}$ is based on the \emph{Alexander function} introduced by Lescop~\cite{L}. The proof of the invariance and the main properties follow from algebraic and homological arguments developed in a paper by the second author and Massuyeau~\cite{FM}.

Ribbon colored tangles can be~\emph{splitted} into morphisms in a category~$\mathsf{R}ib_G$. We show that the invariant~$\mathsf{A}$ induces a functor from~$\mathsf{R}ib_G$ to the category of~$\mathbb{Z}[G]$-graded modules, see Theorem~\ref{ind}. In the case of \emph{ribbon braids}, that are
 the analogues of braids in this context, the functor coincides with the exterior powers of the~\emph{ad-hoc} colored Burau-Gassner representation, multiplied by a certain relative Alexander polynomial, see Proposition~\ref{burau}.

The multiplicativity of~$\mathsf{A}$ fits naturally in the context of planar algebras, introduced by Jones~\cite{J}.
We construct an algebra $\mathsf{C}ob_G$ over the colored~\emph{cobordisms} operad, see Theorem~\ref{cob}. The cobordisms are smooth cylinder in  $B^4$ with a finite collection of disjoint ``small'' balls removed, and the compositions are given by identification of some  boundary components. We observe that cobordisms act on tangles and that~$\mathsf{A}$ commutes with this action, see~Theorem~\ref{morph}.

Ribbon knotted objects admit representations as \emph{broken surfaces} via projections on the $3$-dimensional space, and through \emph{welded diagrams} which are representations in the $2$-dimensional space. These representations were first introduced by Yajima~\cite{Yaj} and  Satoh~\cite{S}. A complete description is given in the papers by Audoux, Bellingeri, Meilhan and Wagner~\cite{ABMW, Au}. This offers perspectives from both point of view: diagrams permit combinatorial computations of invariants of ribbon knotted objects, and ribbon knotted objects provide topological realizations of welded diagrams.
Moreover, a presentation of the group of ribbon knotted objects is obtained from the diagrams by a Wirtinger type algorithm, see~\cite{ABMW}. 

We give a diagrammatic description of the invariant~$\mathsf{A}$, see Theorem~\ref{weld} and show that it extends the multivariable Alexander polynomial of virtual tangles developed by Archibald and Bar-Natan. For futher details on these objects, see also~\cite{BS} and~\cite{P}.
 The  algebra~$\mathsf{C}ob_G$ appears as a topological incarnation of a circuit algebra of diagrams~$\mathsf{W}eld_\mu$, essentially introduced in~\cite{Arc}.
The compatibility of $\mathsf{A}$ with this structure allows local calculations in the diagram. If these diagrams do not have any virtual crossing, the construction holds for usual links, and we obtain a purely local description of the usual Alexander polynomial. This extends the construction of the Alexander representation by the second author, Bigelow and Cattabriga, see~\cite{BCF}.

  Our construction arises in the context of defining generalizations of Alexander polynomials to tangle-like objects, in which we can include, in addition to~\cite{BCF} concerning usual tangles and~\cite{Arc, P} concerning virtual tangles, also the works of Cimasoni and Turaev, through Lagrangian categories~\cite{CT}, 
Bigelow~\cite{Big} and Kennedy~\cite{K} which studied diagrammatical invariants of usual tangles, Sartori~\cite{Sar}, who defined quantum invariants of framed tangles, and Zibrowius in~\cite{ZIB}. This last one in particular defines an invariant for usual tangles which consists in a finite set of Laurent polynomials, and states, without explicit calculation, that on usual tangles, one can calculate Archibald's invariant from his set of invariants and vice versa.  However, nothing seems to appear in the literature about the $4$-dimensional case of ribbon tangles.

In Section~\ref{topo} we recall the definitions of ribbon tangles and construct the invariant~$\mathsf{A}$. Section \ref{ribbon} is devoted to the  {circuit} algebra $\mathsf{C}ob_G$ over the cobordism operad, and the properties of $\mathsf{A}$ with respect with the action of cobordisms on tangles. In Section~\ref{diagram} we describe the diagrammatic construction of $\mathsf{A}$ and a circuit algebra $\mathsf{W}eld_\mu$ related to~$\mathsf{C}ob_G$. In Section~\ref{ex} we compute some examples.

 \section{The Alexander invariant \texorpdfstring{$\mathsf{A}$}{}}
\label{topo}

In this section $G$ is a free abelian group, and $R$ is the group ring~$\Zz[G]$. 

\subsection{Ribbon tangles}
Let $m$ be a positive integer and $X$  a submanifold of the $m$-dimensional ball~$B^m$. An immersion $Y \subset X$ is \emph{locally flat} if and only if it is locally homeomorphic to a linear subspace $\Rr^k$ in $\Rr^m$ for some~$k \leq m$, except on $\partial X$ and/or $\partial Y$, where one of the $\Rr$ summands should be replaced by~$\Rr_+$.
An intersection $Y_1 \cap Y_2 \subset X$ is \emph{flatly transverse} if  it is locally homeomorphic to the intersection of two linear subspaces $\Rr^{k_1}$ and $\Rr^{k_2}$ in $\Rr^m$ for some positive integers $k_1, k_2 \leq m$ except on~$\partial X$, $\partial Y_1$ and/or~$\partial Y_2$, where one of the $\Rr$ summands is replaced by~$\Rr_+$.

An intersection $D=Y_1 \cap Y_2 \subset S^4$ is a \emph{ribbon disk} if it is homeomorphic to
the $2$-disk and satisfies: $D \subset \mathring{Y}_1$, $\mathring{D} \subset \mathring{Y}_2$ and $\partial D$ is an essential curve in~$\partial Y_2$. More details on ribbon knotted objects can be found in~\cite{ABMW, Au}. 

\begin{de} 
\label{ribtan}
Let $L$ be an oriented trivial link with $2n$ components in~$S^3= \partial B^4$.
A \emph{ribbon tangle} $T$ is a locally flat proper embedding  in $B^4$ of oriented disjoint annuli $S^1 \times I$ denoted $A_1,\dots,A_n$ and disjoint tori $S^1 \times S^1$ denoted $E_1,\dots, E_m$ such that:
\begin{enumerate}[label=\roman*)]
\item There exist locally flat immersed solid tori $F_i$ for $i=1,\dots,m$ such that $\partial F_i = E_i$.
\item $\partial A_i \subset L $ and the orientation induced by $A_i$ on $\partial A_i$ coincides with the given orientation of the two components of~$L$.
\item There exists $n$ locally flat immersed $3$-balls $B_i \simeq B^2 \times I$ whose singular set are finite number of ribbon disks and such that, for all $i \in \{1, \dots, n \}$:
\[
\partial B_i = A_i \cup_\partial (B^2 \times \{0,1 \}).
\] 
\end{enumerate}
\end{de}

A $G$-\emph{colored ribbon tangle} is a pair $(T,\varphi)$ where $T$ is a ribbon tangle with complement~$X_T=B^4 \setminus T$, equipped with a group homomorphism~$\varphi\colon H_1(X_T) \to G$.  
\subsection{Definition of \texorpdfstring{$\mathsf{A}$}{}}
\label{upsilon}

Let $(X,Y)$ be a pair of topological spaces. Denote $p: \hat{X} \rightarrow X$ the maximal abelian cover. For a ring homomorphism
 $\varphi: \mathbb{Z}[H_1(X)] \rightarrow R$, we define the twisted chain complex 
\begin{equation*}
C^\varphi(X, Y)= C(\hat{X}, p^{-1}(Y)) \otimes_{\Zz[H_1(X)]} R 
\end{equation*}
whose homology is denoted~$H_*^\varphi(X,Y;R)$, or simply~$H_*^\varphi(X,Y)$.

Let $(T,\varphi)$ be a $G$-colored ribbon tangle. 
The homomorphism $\varphi$ extends to a ring homomorphism $\varphi \colon \mathbb{Z} [H_1(X_T)] \to R$. 
 For the rest of this section, we set $H=H_1^\varphi(X_T,\ast)$.

\begin{propo} 
\label{pres}
The $R$-module $H$  admits a  presentation with deficiency~$n$.
\end{propo}
We postpone the proof of Proposition~\ref{pres} to the end of Section~\ref{SS:AlexMatrix}.
Consider a presentation of $H$ of the form:
\begin{equation*}
\label{E:presentation}
H  = \langle \gamma_1, \dots, \gamma_{n+q} \mid r_1, \dots, r_q \rangle.
\end{equation*}
Let $\Gamma$ be the free $R$-module generated by~$\langle \gamma_1, \dots, \gamma_{n+q} \rangle$: the relators $r_1, \dots, r_q$ are words in these generators in $\Gamma$. Let us denote $r = r_1 \wedge \cdots \wedge r_q$ and $\gamma=\gamma_1 \wedge \cdots \wedge \gamma_{n+q}$.
The Alexander function $\mathcal{A}_{T}^\varphi \colon \largewedge^n H \to R$ is the $R$-linear map defined by
\[
 r \wedge \tilde{u} = \mathcal{A}_{T}^\varphi (u) \cdot \gamma
\]
for all $u =u_1 \wedge \cdots \wedge  u_n \in \largewedge^n H$, where $\tilde{u}_1,\cdots,\tilde{u}_n$ are arbitrary lifts in $\Gamma$ and $\tilde{u}= \tilde{u}_1 \wedge \cdots \wedge \tilde{u}_n$. Different $n$-deficient presentations will give rise to Alexander functions that differ only by multiplication by a unit in~$R$. Note that if $H$ is free of rank~$n$, then $\mathcal{A}_{T}^\varphi$ is a volume form. 

Let us consider the $q \times (q+n)$ matrix defined by the presentation of~$H$. If one adds to this matrix the row vectors giving $u_1, \dots , u_n$ in the generators $\gamma_1, \dots, \gamma_{q+n}$, then $\mathcal{A}_{T}^\varphi(u)$ is the determinant of the resulting $(q+n) \times (q+n)$ matrix.

\begin{ex} \label{ex1}
Suppose that $G$ has rank $2$ and is generated by $t_1,t_2$. Consider the module $H$ whose presentation has generators~$\gamma_1,\dots,\gamma_4$ and two relations given by the matrix:
\[
\begin{pmatrix}
 -1 & 0 & 1  & 0  \\
 0 & -1 & 1-t_1  & t_2  \\
\end{pmatrix}
\] 

The values of the Alexander function $\mathcal{A}\colon \largewedge^2 H \to R$ are 
\begin{gather*}
 \mathcal{A}(\gamma_1 \wedge \gamma_2)= t_2, \ \mathcal{A}(\gamma_1 \wedge \gamma_3)=0, \ \mathcal{A}(\gamma_1 \wedge \gamma_4)=1,\\
 \mathcal{A}(\gamma_2 \wedge \gamma_3)= -t_2 , \ \mathcal{A}(\gamma_2 \wedge \gamma_4)= t_1-1 , \ \mathcal{A}(\gamma_3 \wedge \gamma_4)= 1.
\end{gather*}
\end{ex}

Let $H_\partial=H_1^\varphi(S^3 \setminus L, \ast)$, which is the free $R$-module of rank~$2n$, generated by the meridians of~$L$.
Let $m_\partial\colon H_\partial \to H$ be induced by the inclusion map~$S^3 \setminus L \hookrightarrow X_T$. For short, for a given $z \in \largewedge^n H_\partial$, we use the notation $m_\partial z$ for $\wedge^n m_\partial (z)$. 

\begin{de} \label{alex}
The element $\mathsf{A}(T,\varphi)$ of $\largewedge^n H_\partial$ is the (colored) isotopy invariant defined by the following property:
\begin{equation} \label{deftangle}
\forall z \in \largewedge^n H_\partial,  \  \mathcal{A}_{T}^\varphi (m_\partial z)= \omega_\partial (\mathsf{A}(T,\varphi) \wedge z)
\end{equation}
where $\omega_\partial$ is a volume form on~$H_\partial$.
\end{de}

\subsection{The Alexander polynomial of a \texorpdfstring{$(1-1)$-tangle}{}} \label{Al}

Given a finitely generated $R$-module~$H$, and $k \geq 0$, the \emph{$k$-th Alexander polynomial} of $H$ is the greatest common divisor of all minors of order $(m-k)$ in a $q \times m$ presentation matrix of~$H$. This invariant of~$H$, denoted $\Delta_k (H) \in R$, is defined up to multiplication by a unit of~$R$.  

\begin{de} 
The \emph{Alexander polynomial} $\Delta^\varphi(T) \in R$ of a $G$-colored ribbon tangle $(T,\varphi)$ is $\Delta_0 \left( H_1^\varphi(X_T) \right)$. 
\end{de}

Similarly to classical knot theory, $T$ is a $(1-1)$-\emph{tangle} if~$n=1$. The components of $T$ in $B^4$ consist of $m$ tori and a cylinder whose boundary is a $2$-component trivial link $L$ in~$S^3 = \partial B^4$. Let $x_1$ and $x_2$ be the meridians of the components of~$L$.
Note that in~$X_T$, both $x_1$ and $-x_2$ are homologous to the meridian $x$ of the cylinder. We use the same notations $x_1$ and $x_2$ for the homology classes of their lifts in~$H_1^\varphi(S^3 \setminus L,\ast)$.

\begin{propo} \label{oneone}
Let $(T,\varphi)$ be a $G$-colored $(1-1)$-tangle, such that $\varphi$ is not trivial. Denote~$t=\varphi(x)$ and let~$r$ be the rank of~$\varphi(H_1(X_T))$.
Then the element $\mathsf{A}(T,\varphi)$ of $H_\partial$ is given by 
\[
\mathsf{A}(T,\varphi)= 
\begin{cases}
(t-1) \Delta^\varphi(T) \cdot (x_1-x_2) & \hbox{if } r \geq 2,\\
 \Delta^\varphi(T) \cdot (x_1-x_2) & \hbox{if } r=1.
\end{cases}
\]
\end{propo}

It is worth noticing that, up to a unit in~$R$, the result is independent of the order chosen on the components of~$L$.
For the reader's convenience, we give a short proof of Proposition~\ref{oneone}. More detailed arguments can be found in~\cite[Section 3]{FM}.

\begin{proof} 
Denote~$H=H_1^\varphi(X_T,\ast)$.
From the long exact sequence of the pair~$(X_T,\ast)$:
\[
\begin{CD}
0    @>>>   H_1^\varphi(X_T)  @>>> H @>>> H_0^\varphi(\ast) @>>> H_0^\varphi(X_T)  @>>>0\\
\end{CD}
\]

we deduce that $\mathrm{Tors} H_1^\varphi(X_T) \simeq  \mathrm{Tors} H$. Moreover,
 $ \mathrm{rk } H =  \mathrm{rk } H_1^\varphi(X_T) +1$. This implies that
\[\Delta^\varphi(T)= \Delta_1( H).\]
Let $A$ be the matrix of a presentation of $H=\langle \gamma_1, \dots, \gamma_{q+1} \mid r_1, \dots, r_q \rangle$, and $\mathcal{A}^\varphi$ be the related Alexander function. We have
\[
\forall z_1,\ldots,z_{q+1} \in R,\ \mathcal{A}^\varphi(z_1 \gamma_1 + \cdots + z_{q+1} \gamma_{q+1}) = \sum_{i=1}^{q+1} \det(A_i) z_i,
\]
 where $A_i$ is $A$ with the $i$-th column removed.
Hence, 
\[
\Delta_1(H) = \gcd \mathcal{A}^\varphi(H)= \gcd \{ \mathcal^\varphi(h); \, h \in H \}.
\]
 If $\mathcal{A}^\varphi=0$, then~$\Delta_\varphi(T)=0$.
Consider now the connecting homomorphism~$\partial_\ast \colon H \to H_0^\varphi(\ast) \simeq R$.
 If $\mathcal{A}^\varphi \neq 0$, then rank $(H)=1$ and two linear maps
 $H \otimes_R \mathbb{Q} R \to \mathbb{Q} R$ are linearly dependent, where $\mathbb{Q} R$ is the fraction field of~$R$. 
 Then there exist elements~$P,Q$ in $R$ such that for each~$h \in H$,  $\mathcal{A}^\varphi(h)= (P/Q) \partial_\ast(h)$. Hence
\[\mathcal{A}^\varphi(h)= \Delta^\varphi(T) \cdot \frac{\partial_\ast (h)}{\gcd \partial_\ast(H)}.\]
For a loop $\gamma$ based in $\ast$ with lift $\hat{\gamma}$, one has $\partial_\ast \hat{\gamma}= \varphi(\gamma) -1$. Hence,
$
\gcd \partial_\ast(H)$ is equal to $1$ if $r \geq 2$ and is equal to $t-1$ if~$r =1$.
We deduce that for all $h \in H$
\[
\mathcal{A}^\varphi(h) =
\begin{cases}
(t-1) \Delta^\varphi(T) & \mbox{ if }  r \geq 2,\\
 \Delta^\varphi(T)  & \mbox{ if }  r =1.
\end{cases}
\]
Let $\omega_\partial$ be the volume form on $H_\partial$ relative to the choice of the meridians~$x_1,x_2$.
By definition, $\mathsf{A}(T,\varphi)$ verifies $\mathcal{A}_T^\varphi(m_\partial z)= \omega_\partial(\mathsf{A}(T,\varphi) \wedge z)$. Since $m_\partial(x_1)=x$ and $m_\partial(x_2)= -x$, we obtain the result from 
\[
\mathcal{A}^\varphi(x)= \omega_\partial(\mathsf{A}(T,\varphi) \wedge x_1)= -\omega_\partial(
 \mathsf{A}(T,\varphi) \wedge x_2).
\]
\end{proof}

\subsection{The Burau functor}
\label{SS:Burau}

Let  $\mathcal{R}ib$ be the category of ribbon cobordisms.
The objects are sequences $\varepsilon$ of signs $\pm 1$ of length~$n$, and correspond to trivial links with $n$ components in~$B^3$, such that a sign is affected to each component.
 The morphisms $\varepsilon_0 \to \varepsilon_1$
 are the equivalence classes of ribbon cobordisms between $L_0$ and~$L_1$. More precisely, a \emph{ribbon cobordism} from $\varepsilon_0$ to $\varepsilon_1$ is a  collection $S$  of  ribbon annuli and tori in $B^3 \times I$ whose boundaries are the components of $L_0$ and~$L_1$, with according signs. Two cobordisms are equivalent if there is an ambient isotopy fixing the boundary circles $L_0$ and~$L_1$.  The \emph{degree} of the ribbon cobordism $S$ is~$\delta=(n_1 - n_0)/2$. The composition $S \circ S'$ of two ribbon cobordisms $S$ and $S'$ in $\mathcal{R}ib$ is defined by identifying $(B^3 \times \{1 \}, S')$ to~$(B^3 \times \{0 \}, S)$.

The category $\mathcal{R}ib$ can be refined to the category of colored ribbon cobordisms~$\mathcal{R}ib_G$.
 An object is a pair $(\varepsilon, \varphi)$, where $\varphi\colon H_1(B^3 \setminus L) \to G$. A morphism $(\varepsilon_0,\varphi_0) \to (\varepsilon_1,\varphi_1)$ is a pair $(S,\varphi)$ such that $\varphi \circ m_i= \varphi_i$ for~$i=1,2$, if $m_i$ are induced by the inclusions~$X_{L_i} \hookrightarrow X_S$, where $X_{L_i}= (B^3 \times \{i\}) \setminus L_i$ and~$X_S= (B^3 \times I) \setminus S$. 

 We define a projective functor $\rho$ from the category $\mathcal{R}ib_G$  to the category of $\mathbb{Z}$-graded $R$-modules~$gr\mathcal{M}od_G$,
whose morphisms are graded $R$-linear maps of arbitrary degree, up to multiplication by
an element of~$\pm G$.
Let $(\varepsilon,\varphi)$ be an object of~$\mathcal{R}ib_G$. The image of $(\varepsilon,\varphi)$ by $\rho$ is
\[ \rho(\varepsilon,\varphi)= \largewedge H_1^\varphi(X_L,\ast;R), \]
 the exterior algebra of the free $R$-module $H_1^\varphi(X_L,\ast;R)$ of rank~$n$, where $\ast$ is a base point in~$\partial B^3$. 
  Next, to a morphism $(S,\varphi)\colon (\varepsilon_0,\varphi_0) \to (\varepsilon_1,\varphi_1)$, we associate 
 a $R$-linear map \[\rho(S,\varphi)\colon \largewedge  H_1^{\varphi_0}(X_{L_0} ,\ast;R) \to \largewedge  H_1^{\varphi_1}(X_{L_1},\ast;R),\]
 of degree $\delta$ as follows.
Let $J$ be the interval in $\partial B^3 \times I$, which connects the base points of the bottom and top balls. 
We denote the $R$-modules $M_0=H_1^{\varphi_0}(X_{L_0},\ast;R)$ and~$M_1=H_1^{\varphi_1}(X_{L_1},\ast;R)$.  The module $H_1^\varphi(X_S,J;R)$ admits a presentation with deficiency $d=(n_0 + n_1)/2$. Denote the Alexander function $\mathcal{A}_S^\varphi  \colon \largewedge^{d} H_1^\varphi(X_S,J;R) \to R$.
 For any integer~$k \geq 0$, the image  $\rho(S,\varphi)(x) \in \largewedge^{k+ \delta}  M_1$ of any $x \in \largewedge^k  M_0$ is defined by the following property
\[ \forall y \in \largewedge^{d-k} M_1, \ \mathcal{A}_S^\varphi(\largewedge^k m_0 (x) \wedge
 \largewedge^{d-k} m_1(y)) = \omega_1 \left( \rho(S,\varphi)(x) \wedge y \right),\]
where $\omega_1$ is a volume form on~$M_1$.

\begin{thm} \label{funct}
The map $\rho$ is a degree preserving functor $\mathcal{R}ib_G \to gr\mathcal{M}od_G$. 
\end{thm}
The proof of Theorem~\ref{funct} follows word by word the proof of~\cite[Theorem I]{FM}. Note that the deficiency of the presentation of $H_1^\varphi(X_S,J;R)$ depends only on $n_0$ and~$n_1$.

A $G$-colored \emph{ribbon tube} $(S,\varphi)$ is a morphism such that the inclusions $m_1$ and $m_2$ induce isomorphisms in homology (with integer coefficients). Ribbon tubes are analogous to string links; the links $L_0$ and $L_1$ have the same number of components, and $S$ has no toric component. For a fixed~$\varphi$, the set of $G$-colored ribbon tubes form a monoid~$\mathcal{T}_\varphi$. 
Following~\cite[Proposition 2.1]{KLW}, one proves that $m_1$ and $m_2$ induce isomorphisms
\[ (m_i)_\ast \colon H^\varphi(X_{L_i}; QR) \longrightarrow H^\varphi(X_{S}; QR), \,  \mbox{for } i=1,2 
\]
  where $QR$ is the quotient field of~$R$. Set $H^\varphi$ to be~$H^\varphi(X_{L_1}; QR)=H^\varphi(X_{L_2}; QR)$. Set $L=L_0=L_1$.
 The composition $m_2^{-1} \circ m_1$ is an automorphism of~$H^\varphi(X_{L}; QR)$.

\begin{de} 
The \emph{colored Burau representation} is the monoid homomorphism \[r^\varphi \colon\mathcal{T}_\varphi \to  \mathrm{Aut}(H^\varphi).\]
\end{de}

Let $\Delta^\varphi(X_T,X_{L_2})= \Delta_0 \left( H_1^\varphi(X_T, X_{L_2}; R) \right)$ be the Alexander polynomial of the pair $(X_T,X_{L_2})$, see Section \ref{Al}.

\begin{propo} \label{burau}
For any $G$-colored ribbon tube $(S,\varphi) \in \mathcal{T}_\varphi$, we have
 \[ \rho(S,\varphi) = \Delta^\varphi(X_T,X_{L_2}) \cdot  \largewedge r^\varphi \, \colon \largewedge H^\varphi \to \largewedge H^\varphi.\]
\end{propo}

In particular, if the ribbon tube $(S,\varphi)$ is \emph{monotone} with respect to the coordinate in $I$ in $B^4 \times I$, it is called a \emph{ribbon braid} (see~\cite{ABMW, Dam}). In this case $\Delta^\varphi(X_T,X_{L_2})=1$ and $\rho(S,\varphi)$ coincide with the exterior powers of~$r^\varphi$.
The proof of Proposition \ref{burau} can be obtained similarly to~\cite[Proposition 7.2]{FM}, or in the monotone case, to~\cite[Section 3.1]{BCF}.

Let $(T,\varphi)$ be a $G$-colored ribbon tangle in~$B^4$. Let $L= L_0 \cup L_1$ be a splitting of $L$ into two disjoint (trivial) links, and let $B_0$ and $B_1$ be two $3$-balls such that 
\[
S^3= B_0 \cup_{S^2 \times \{ 0 \} } S^2 \times [0;1] \cup_{S^2 \times \{ 1 \} }  B_1
\]
 and $L_i \subset B_i$ for~$i=1,2$. 
 Let $\varphi_i$ be induced by $\varphi$ on $H_1(B_i \setminus L_i)$, and $\varepsilon_i$ be sequences of signs according to the co-orientations of the components of~$T$. 
 Then $(T,\varphi)$ is \emph{splitted} as a morphism $(\tilde{T},\varphi)\colon (\varepsilon_0,\varphi_0) \to (\varepsilon_1,\varphi_1)$. Note that~$2n=n_0+n_1$.

\begin{thm} \label{ind}
Let $(T,\varphi)$ be a $G$-colored ribbon tangle, and $(\tilde{T},\varphi)\colon (\varepsilon_0,\varphi_0) \to (\varepsilon_1,\varphi_1)$ be a splitting of $(\mathcal{T},\varphi)$ in~$\mathcal{R}ib_G$. 
There is an isomorphism, well-defined up to a unit in~$R$,
\[ \largewedge^n H_\partial \to  \mathrm{Hom}_R (\largewedge M_0,\largewedge M_1)\]
sending $\mathsf{A}(T,\varphi)$ to $ \oplus_k (-1)^{k(n_0-k)} \rho_k (\tilde{T},\varphi)$, where $\rho_k$ is the $k$-component of~$\rho$.
\end{thm}

An explicit example is given in Remark \ref{bur}.

\begin{proof}
The decomposition $H_\partial = M_0 \oplus M_1$ induces a natural isomorphism \[ \largewedge^n H_\partial \simeq \bigoplus_{k=0}^n \big( \largewedge^k M_0 \otimes \largewedge^{n-k} M_1 \big).\]
The element $\mathsf{A}(T,\varphi) \in \largewedge^n H_\partial$ decomposes as $\sum_k \mathsf{A}_k(T,\varphi)$, where $\mathsf{A}_k(T,\varphi) \in \largewedge^k M_0 \otimes \largewedge^{n-k} M_1$. Suppose now that $k$ is fixed; the element $\mathsf{A}_k(T,\varphi)$ might not be decomposable. There exist a finite sequence of element $A^l_0 \in  \largewedge^k M_0$ and $A^l_1 \in  \largewedge^{n-k} M_1$ (depending on $k$, and $(T,\varphi)$) such that 
$\mathsf{A}_k(T,\varphi) = \sum_l A_0^l \otimes A^l_1$. 
Let $\omega_0$ be a volume form  $\largewedge^{n_0} M_0 \stackrel{}  \to R$.
There is an isomorphism
\[  \largewedge^k M_0 \otimes \largewedge^{n-k} M_1 \simeq  \mathrm{Hom}_R(\largewedge^{n_0-k} M_0, \largewedge^{n-k} M_1)\] 
sending $\mathsf{A}_k(T,\varphi)$ to 
\[x \mapsto \sum_l \omega_0 \left(x \wedge A_0^l \right) \cdot A_1^l.\]
We now show that this morphism coincides with~$\rho_{n_0-k} (\tilde{T},\varphi)$. Let $x \in \largewedge^{n_0-k} M_0$. Note that we have $n_0-k + \delta= n-k$ and the morphism has degree~$\delta$. Consider  a volume form  $\omega_1\colon \largewedge^{n_1} M_1 \to R$ and the sum $\omega_\partial= \omega_0 \otimes \omega_1 \colon \largewedge^n H_\partial \to R$. Let $\mathcal{A}^\varphi$ be the Alexander function related to a presentation of $H_1^\varphi(X_T,\ast;R)$ of deficiency~$n$.
Since $m_\partial = m_0 \oplus m_1$ we have, for all $y $ in $\largewedge^{k+\delta}M_1$,
\begin{dmath*}
 \mathcal{A}^\varphi ( \largewedge^{n_0-k} m_0(x) \wedge \largewedge^{k+\delta} m_1(y) )  =
 \omega_\partial ( \mathsf{A}(T,\varphi) \wedge x \wedge y ) 
 = \omega_\partial ( \mathsf{A}_k(T,\varphi) \wedge x \wedge y ) 
 = \sum_l \omega_0 ( A_0^l  \wedge x ) \cdot \omega_1 (  A^l_1 \wedge y ) 
 = \omega_1 \big( \sum_l  \omega_0 (A_0^l   \wedge x ) \cdot A^l_1 \wedge y \big).
\end{dmath*}
Hence $\sum_l  \omega_0 (x \wedge A_0^l ) \cdot A^l_1= (-1)^{k(n_0-k)}\rho_{n_0 -k}(\tilde{T},\varphi)(x)$. Note that if, for cosmetic reasons, one prefers that $\mathsf{A}_k(T,\varphi)$ induces $(-1)^{k(n_0-k)}\rho_{k}(\tilde{T},\varphi)(x)$ instead, one might compose with Hodge duality $\largewedge^{k} M_0 \rightarrow \largewedge^{n_0- k} M_0$.
\endproof

\section{The circuit algebra of colored cobordisms}
\label{ribbon}

In this section, we introduce the circuit algebra~$\mathsf{Cob}_G$. The algebraic structure is inspired by~\cite{J, P},   see also~\cite{K, Arc}.

\begin{de}
Let $B=B_0$ be a $4$-ball and  $B_1,\dots,B_p$ be disjoint $4$-balls in the interior of~$B$.
For every $i \in \{0,\ldots,p\}$, let $L_i$ (with $L=L_0$) be a trivial oriented link with $2n_i$ ($n=n_0$) components in $S^3_i=\partial B_i$ (with $S^3=S^3_0$).
A \emph{cobordism} $C$ is a disjoint union of locally flat proper embedded annuli in $B \setminus \{  \mathring{B}_1,\dots \mathring{B}_p\}$, whose boundary are the links~$L_i$, with the conditions of Definition~\ref{ribtan} but without singularities. 
\end{de}

\begin{de}
Let $C'$ and $C''$ be two cobordisms such that $B'_i$ is a ball of $C'$ with~$n'_i=n''$. The  \emph{composition} $C' \circ_i C''$
 is the cobordism obtained with the identification of $B''=B_0''$ with~$B'_i$.
\end{de}

As in the previous section, $G$ is a fixed free abelian group with group ring~$R$.
 A $G$-\emph{colored} cobordism is a pair $(C,\varphi)$ 
 where $C$ is a cobordism with complement~$X_C$, equipped with is a group homomorphism $\varphi\colon H_1(B^4 \setminus C) \to G$. 
The orientation-preserving diffeomorphism classes of $G$-colored cobordisms with composition of compatible cobordisms, form an operad denoted~$\mathcal{C}_G$.

Let $\mathcal{H}om_G$ be the operad of tensor powers of $R$-modules and $R$-linear applications, considered up to an element of~$\pm G$. The composition in $\mathcal{H}om_G$ is induced by the usual composition of maps.
We construct the circuit algebra $\mathsf{Cob}_G$ as a morphism of operads from  $\mathcal{C}_G$ to~$\mathcal{H}om_G$.

Let $(L,\varphi)$ be a $G$-colored oriented trivial link with $k$ components in~$S^3$, with complement $X_L=S^3 \setminus L$. The group homomorphism $\varphi \colon H_1(S^3 \setminus L) \to G$ induces a ring homomorphism denoted also $\varphi\colon \mathbb{Z}[H_1(S^3 \setminus L)] \to R$.
Let $\ast$ be a base point on~$S^3$.    
 The $R$-module $H_1^\varphi(S^3 \setminus L, \ast; R)$ is free of rank~$k$, generated
 by the meridians of~$L$.

Let now $(C,\varphi)$ be a $G$-colored cobordism, with complement~$X_C$.
For $i=1,\dots,p$, let $\ast$ and $\ast_i$ be base points in the boundary of $B_i$ and $J_i$ be intervals (whose interiors are disjoint, and disjoint from $C$) connecting $\ast$ to~$\ast_i$. 
Note that the union of the $J_i$ is contractible. The homomorphism $\varphi$ induces a ring homomorphism $\mathbb{Z}[H_1(X_C)] \to R$ also denoted~$\varphi$. The inclusion $m_i\colon S^3_i \setminus L_i \hookrightarrow X_T$ induces $\varphi_i \colon \mathbb{Z}[H_1(S^3_i \setminus L_i)] \to R$. Set
 $H= H_1^\varphi(X_C,J)$, $H_\partial=H_1^\varphi(S^3\setminus L, \ast)$, and $H_{\partial_i}= H_1^{\varphi_i}(S^3 \setminus L_i, \ast_i)$ for~$i=1,\dots,p$. Note that $H$ is free of rank $ r=n+n_1+\cdots+n_p$ and $H_{\partial_i}$ are free of rank~$n_i$.
Let $\omega_C^\varphi$ be a volume form $\omega_C^\varphi\colon \largewedge^{r} H \to R$,
 and $\omega_\partial \colon \largewedge^{2n} H_\partial \to R$. For $i=1,\dots,p$, we denote again $m_i \colon H_{\partial_i} \to H$ the map induced by the inclusion. Let $\displaystyle{m \colon \otimes_{i} \largewedge^{n_i} H_{\partial_i} \to
 \largewedge ^{n_1+\cdots+n_p} H}$ be defined as $\displaystyle m= \largewedge^{n_1} m_1 \wedge \cdots \wedge \largewedge^{n_p} m_p$. 

 To the cobordism $(C,\varphi)$ we associate \[ \Upsilon_{C,\varphi}\colon \bigotimes_{i=1}^p \largewedge^{n_i} H_{\partial_i}
 \to \largewedge^{n} H_\partial\] such that, for $x \in \otimes_{i} \big( \largewedge^{n_i} H_{\partial_i} \big)$,
\begin{equation} \label{defcob}
 \omega_C^\varphi ( m(x) \wedge m_\partial(y) ) = \omega_\partial ( \Upsilon_{C,\varphi}(x) \wedge y), \ \forall y \in \largewedge^n H_\partial.
\end{equation}

We denote $\mathsf{Cob}_G$ the object that associates to a pair $(L,\varphi)$ the homology module~$H_1^\varphi(S^3 \setminus L, \ast; R)$, and to each pair $(C,\varphi)$ the linear map~$\Upsilon_{C,\varphi}$.


\begin{thm} \label{cob}
$\mathsf{Cob}_G$ is a circuit algebra.
\end{thm}

\begin{proof}
We show that $\Upsilon$ commutes with the composition of cobordisms.
Let $(C',\varphi')$ and $(C'',\varphi'')$ be two cobordisms with 
\[
\Upsilon_{C',\varphi'}\colon \bigotimes_{k=1}^{p'} \largewedge^{k'_i} H_{\partial'_k}
\to \largewedge^{n'} H_\partial', \quad \mbox{and} \quad \Upsilon_{C'',\varphi''}\colon \bigotimes_{l=1}^{p''} \largewedge^{n''_l} H_{\partial''_l}
\to \largewedge^{n''} H_{\partial''}.
\]

Let $\varphi$ be the coloring induced by $\varphi'$ and $\varphi''$ on~$C' \circ_i C''$.
Then, we have to prove that for all $u_k  \in \largewedge^k_iH_{\partial'_k}$ with $k=1,\dots,p'$ and $k \neq i$ and all $v_l$ for~$l=1,\dots,p''$,
\[ \Upsilon_{C' \circ_i C'',\varphi} \big(u_1 \otimes \cdots \otimes (\otimes_{1}^{p''} v_l) \otimes \cdots \otimes u_{p'} \big)
 = \Upsilon_{C',\varphi'} \big( u_1 \otimes \cdots \otimes \Upsilon_{C'',\varphi''}(\otimes_{1}^{p''} v_l) \otimes \cdots \otimes u_{p'} \big).\]

Let $H'=H_1^{\varphi'}(X_{T'},J')$ and $H''=H_1^{\varphi''}(X_{T''},J'')$ be the (free) homology modules of the exteriors of the cobordisms. Let $\alpha_1,\dots,\alpha_{2n''}$ be a basis of $H_{\partial''} \simeq H_{\partial'_i}$. Consider presentations of $H'$ and~$H''$:
\begin{gather*} 
H'' = \langle m_{\partial''}\alpha_1, \dots, m_{\partial''}\alpha_{2n''}, \beta_1,\dots,\beta_k  \, \vert\, \rho_1,\dots,\rho_s \rangle, \\
H' = \langle m_i' \alpha_1, \dots, m_i'\alpha_{2n''}, \zeta_1,\dots,\zeta_l  \, \vert\, r_1,\dots,r_t \rangle.
\end{gather*}
Applying Mayer-Vietoris theorem to $X_{T}= X_{T'} \cup X_{T''}$, we obtain that the (free) module $H$ is generated by 
\begin{equation}\label{gen}
 m_{\partial''}\alpha_1, \dots, m_{\partial''}\alpha_{2n''},  m_i' \alpha_1, \dots, m_i'\alpha_{2n''}, \beta_1,\dots,\beta_k, \zeta_1,\dots,\zeta_l 
\end{equation}
subject to the relations $\rho_1,\ldots,\rho_s, r_1,\ldots,r_t, m_{\partial''}\alpha_1- m_i' \alpha_1, \ldots, m_{\partial''}\alpha_{2n''}-m_i'\alpha_{2n''}$.
Let $\omega'$ and $\omega''$ be volume forms on $H'$ and~$H''$, and $\omega_{\partial'}$ be the form on~$H_{\partial'}$. Let $\omega$ be the form on $H$ induced by $\omega'$ and~$\omega''$.
 For the computation below, we introduce the notation 
\begin{equation*}
u \wedge_i v = u_1 \wedge \cdots \wedge u_{i-1} \wedge (v_1 \wedge \cdots \wedge v_{p''}) \wedge
 u_{i+1} \wedge \cdots \wedge u_{p'}.
\end{equation*}
We want to show that, for all $y \in  \largewedge^{n'} H_{\partial'}$,
\begin{dmath*}  \omega_{\partial'} \big( \Upsilon_{C' \circ C''}  (u_1 \otimes \cdots \otimes (\otimes_{1}^{p''} v_l) \otimes \cdots \otimes u_{p'} ) \wedge y \big)
 = \omega_{\partial'} \big( \Upsilon_{C'} ( u_1 \otimes \cdots \otimes \Upsilon_{C''}(\otimes_{1}^{p''} v_l) \otimes \cdots \otimes u_{p'} ) \wedge y \big).
\end{dmath*}

We have
\begin{dmath*}
 {\omega_{\partial'} \big( \Upsilon_{C' \circ C''}  (u_1 \otimes \cdots  \otimes  (\otimes_{1}^{p''} v_l) \otimes \cdots \otimes u_{p'} ) \wedge y )\cdot m_{\partial''} \alpha \wedge \beta \wedge m'_i \alpha \wedge \zeta} 
  = {  \omega( (m'u \wedge_i m''v) \wedge m_{\partial'} y) \cdot m_{\partial''} \alpha \wedge \beta \wedge m'_i \alpha \wedge \zeta}  
  = { \rho \wedge r \wedge (m_{\partial''}\alpha- m_i' \alpha) \wedge (m'  u \wedge_i  m''v ) \wedge m_{\partial'} y}
   = {  \sum_Q (-1)^{|Q|} \epsilon_Q \cdot \rho \wedge r \wedge m_{\partial''} \alpha_Q \wedge m'_i \alpha_{\bar{Q}} \wedge (m' u \wedge_i m''v) \wedge m_{\partial'} y,}
\end{dmath*}
where the sum is taken over all subsets $Q \subset \{1,\dots,2 n'' \}$ of cardinal $n''$. The number $t$ of relations $r_i$ can be chosen arbitrarily to be even and $m_{\partial''} \alpha_Q \wedge m'_i \alpha_{\bar{Q}} \wedge (m' u \wedge_i m''v)$ coincides with $m_{\partial''}\alpha_Q \wedge m''v  \wedge (m' u \wedge_i m'_i \alpha_{\bar{Q}} )$ up to a sign depending only on $p''$ and the $n_i'$, the sum coincides up to a sign with
\begin{equation*}
 \sum_Q (-1)^{|Q|} \epsilon_Q  \cdot \big( \rho \wedge m_{\partial''}\alpha_Q \wedge m''v ) \wedge ( r \wedge (m' u \wedge_i m'_i \alpha_{\bar{Q}} ) \wedge m_{\partial'} y)
\end{equation*}
whose summands are equal to
\begin{equation*}
(-1)^{|Q|} \epsilon_Q \cdot \omega''(m_{\partial''}\alpha_Q \wedge m''v  ) \cdot \omega' ((m' u \wedge_i  m_i' \alpha_{\bar{Q}})   \wedge m_{\partial'} y) \cdot (m_{\partial''} \alpha \wedge \beta \wedge m_i' \alpha \wedge \zeta).
\end{equation*}
It follows that, up to a sign, 
\begin{dmath*}
 {\omega_{\partial'} \big( \Upsilon_{C' \circ C''}  (u_1 \otimes \cdots \otimes (\otimes_{1}^{p''} v_l) \otimes \cdots \otimes u_{p'} ) \wedge y \big) }
 = { \sum_Q (-1)^{|Q|} \epsilon_Q \cdot \omega''(m_{\partial''} \alpha_Q \wedge m'' v ) \cdot \omega' \big((m'u \wedge_i  m_i' \alpha_{\bar{Q}} ) \wedge m_{\partial'} y \big) }
= {  \omega' \big( \sum_Q (-1)^{|Q|} \epsilon_Q  \cdot \omega''(m_{\partial''} \alpha_Q \wedge m'' v ) \cdot (m'u \wedge_i m_i' \alpha_{\bar{Q}})  \wedge m_{\partial'} y \big) }
= { \omega' \big( \sum_Q (-1)^{|Q|} \epsilon_Q \cdot \omega_{\partial''}(\Upsilon_{C''}(v_1 \otimes \cdots \otimes v_{p''} )  \wedge \alpha_Q) \cdot (m'u \wedge_i  m_i' \alpha_{\bar{Q}} ) \wedge m_{\partial'} y \big) }
= {  \omega' \big( \big( m' u \wedge_i \big[ \sum_Q (-1)^{|Q|} \epsilon_Q \cdot  \omega_{\partial''}(\Upsilon_{C''}(v_1 \otimes \cdots \otimes v_{p''} )  \wedge \alpha_Q)  \big] m'_i \alpha_{\bar{Q}} \big) \wedge  m_{\partial'} y \big) }
= {  \omega'  \big(  \big(m' u \wedge_i  m_i' (\Upsilon_{C''}(v_1 \otimes \cdots \otimes v_{p''} ))\big) \wedge m_\partial' y \big) }
= { \omega_{\partial'} \big( \Upsilon_{C'} ( u_1 \otimes \cdots \otimes \Upsilon_{C''}(\otimes_{1}^{p''} v_l) \otimes \cdots \otimes u_{p'} ) \wedge y \big).}
\end{dmath*}

\end{proof}

\subsection{Action of cobordisms on tangles}

Given a cobordism $C$ and a collection of tangles $T_1,\cdots,T_p$, one may create a new tangle, if $n(T_i)=n_i$ for all $i=1,\dots,p$, by gluing each $T_i$ into the internal ball $B_i$ of~$C$.
The action of $G$-colored cobordisms on $G$-colored tangles is defined once the colorings coincide on the boundary components.
The following theorem states that the invariant $\mathsf{A}$ respects the structure of circuit algebra~$\mathsf{Cob}_G$.

\begin{thm} \label{morph}
 Let $(T,\psi)$ be the $G$-colored tangle obtained by gluing the $G$-colored tangles $(T_1, \varphi_1), \cdots, (T_p,\varphi_p)$ to a $G$-colored cobordism~$(C,\varphi)$.
The following equality holds
\[ \mathsf{A}(T,\psi) = \mathsf{\Upsilon}_{C,\varphi} \big( \mathsf{A}(T_1,\varphi_1) \otimes \cdots \otimes \mathsf{A}(T_p,\varphi_p) \big) \in \largewedge^n H_\partial.\]
\end{thm}

\begin{proof}
For $i=1,\cdots,p$, consider a presentation of $H_{T_i}= H_1^{\varphi_i} (X_{T_i}, \ast_i)$ of the form
\begin{equation*}
H_{T_i}  = \langle m_{\partial_i} \gamma_1^i, \dots, m_{\partial_i} \gamma_{2n_i}^i, \beta_1^i, \dots, \beta_{k_i}^i  \mid \rho_1^i, \dots, \rho_{s_i}^i \rangle
\end{equation*}
and let $\mathcal{A}_{T_i}^{\varphi_i}$ be the Alexander function related to this presentation.
Let $H_C=H_1^\varphi(X_C, J)$ be the free module, with volume form $\omega_C^\varphi$ associated to the presentation 
\begin{equation*}
H_{C}  = \langle m_{1} \gamma_1^1, \dots, m_{1} \gamma_{2n_1}^1, \dots, m_p \gamma_1^p,\dots,m_p \gamma_{2n_p}^p, \alpha_1,\dots,\alpha_l   \mid r_1, \dots, r_t \rangle.
\end{equation*}
By successive Mayer-Vietoris arguments, the module $H_{T}= H_1^{\varphi}(X_T,\ast)$ admits a presentation with generators of the form $m_{\partial_i} \gamma_j^i$, $m_{i} \gamma_j^i$, for $j=1,\dots,2n_i$ and $i=1,\dots,p$,
 and  $\beta_1^i, \dots, \beta_{k_i}^i, \alpha_1,\dots,\alpha_l$. They are subject to 
the relations $m_{\partial_i} \gamma_j^i - m_{i} \gamma_j^i$, for $j=1,\dots,2n_i$ and $i=1,\dots,p$, and
 $\rho_1^i, \dots, \rho_{s_i}^i$ for $i=1,\dots,p$ and $r_1,\dots,r_t$.
One has $\mathsf{A}(T,\psi) = \Upsilon_{C,\varphi} \big( \mathsf{A}(T_1,\varphi_1) \otimes \cdots \otimes \mathsf{A}(T_p,\varphi_p) \big)$ if and only if
\begin{equation*}  
\forall z \in \largewedge^{n} H_\partial, \   \mathsf{A}(T,\psi) \wedge z = \Upsilon_{C,\varphi} \big( \mathsf{A}(T_1,\varphi_1) \otimes \cdots \otimes \mathsf{A}(T_p,\varphi_p) \big) \wedge z   
\end{equation*}  
By Equation~\eqref{deftangle},  $\omega_\partial(\mathsf{A}(T,\psi) \wedge z)= \mathcal{A}_T^\psi \big(m_\partial  z \big)$. 
For short, we write $\mathsf{A}(T_i)$ for $\mathsf{A}(T_i,\varphi_i)$, for~$i=1,\dots,p$.
By definition of the Alexander function~$\mathcal{A}_T^\psi$,
\begin{dmath} \label{etoile}
{\mathcal{A}_T^\psi   \big( m_\partial   z \big) \cdot m_{\partial_\ast} \gamma \wedge m_\ast \gamma \wedge \beta \wedge \alpha
= \rho \wedge r \wedge (m_{\partial_i}^i - m_i \gamma^i ) \wedge m_\partial  z  }
 =  {\sum_{Q_1,...,Q_p} (-1)^{|Q|} \epsilon_Q \cdot \rho \wedge r \wedge m_{\partial_1} \gamma_{Q_1}^1 \wedge \cdots \wedge m_{\partial_{p}} \gamma_{Q_p}^p \wedge m_{1} \gamma_{\bar{Q}_1}^1 \wedge \cdots \wedge m_{p} \gamma_{\bar{Q}_p}^p  \wedge m_\partial z,}
\end{dmath}
where the sum is taken over the subsets $Q_i \subset \{1,\dots,2 n_i \}$ of cardinal $n_i$ (the other terms vanish).
We denote $\bar{Q}_i$ the complement of $Q_i$, $|Q|=n_1+\cdots+n_p$ and $\epsilon_Q$ the signature of the permutation $Q_1 \cdots Q_p \bar{Q}_1 \cdots \bar{Q}_p$, where the elements of $Q_1 \cdots Q_p$ in increasing order are followed by the elements of $\bar{Q}_1 \cdots \bar{Q}_p$ in increasing order. Moreover, we can decide arbitrarily that the number $t$ of relations in the presentation of $H_C$ is even, and get \eqref{etoile} to be equal to:
\begin{dmath*}
 {\quad \sum_{Q_1,\dots,Q_p}  (-1)^{|Q|} \epsilon_Q \cdot (\rho \wedge m_{\partial_{1}} \gamma_{Q_1}^1 \wedge \cdots \wedge m_{\partial_{p}} \gamma_{Q_p}^p) \wedge (r \wedge m_{1} \gamma_{\bar{Q}_1}^1 \wedge \cdots \wedge m_{p}   \gamma_{\bar{Q}_p}^p  \wedge m_\partial  z) } \\
 =  \sum (-1)^{|Q|}  \epsilon_Q \cdot (\rho \wedge m_{\partial_{1}} \gamma_{Q_1}^1 \wedge \cdots \wedge m_{\partial_{p}} \gamma_{Q_p}^p) \cdot \omega_C^\varphi( m_{1} \gamma_{\bar{Q}_1}^1 \wedge \cdots \wedge m_{p}  \gamma_{\bar{Q}_p}^p  \wedge  m_\partial  z) \cdot (m_\ast \gamma \wedge \alpha). 
\end{dmath*}  
Moreover, since $s_1,\dots,s_p$ can also be supposed even, 
\begin{dmath*}
\rho \wedge m_{\partial_{1}} \gamma_{Q_1}^1 \wedge \cdots \wedge m_{\partial_{p}} \gamma_{Q_p}^p =  (\rho^1 \wedge m_{\partial_{1}} \gamma_{Q_1}^1 )\wedge \cdots \wedge  (\rho^p \wedge m_{\partial_{p}} \gamma_{Q_p}^p) \\
= \mathcal_{A}_{T_1}^{\varphi_1} (m_{\partial_1} \gamma_{Q_1}^1) \cdots   \mathcal_{A}_{T_p}^{\varphi_p} (m_{\partial_p} \gamma^p_{Q_p})
       \cdot m_{\partial_\ast} \gamma \wedge \beta \\
  = \omega_{\partial_1} (\mathsf{A}(T_1) \wedge  \gamma^1_{Q_1} )  \cdots  \omega_{\partial_p} (\mathsf{A}(T_p) \wedge \gamma^p_{Q_p} ) \cdot  m_{\partial_\ast} \gamma \wedge \beta \\
\end{dmath*}

Then $\mathcal{A}_T^\psi \big(m_\partial z \big) $ coincides with the sum over $Q_1 \cdots Q_p$ of the summands
\begin{dmath*}
 { (-1)^{|Q|} \epsilon_Q \cdot  \omega_{\partial_1} (\mathsf{A}(T_1) \wedge  \gamma^1_{Q_1} )  \cdots  \omega_{\partial_p} (\mathsf{A}(T_p) \wedge \gamma^p_{Q_p} ) \cdot \omega_C^\varphi( m_{\bar{Q}_1} \gamma_{\bar{Q}_1}^1 \wedge \cdots \wedge m_{\bar{Q}_p} \gamma_{\bar{Q}_p}^p  \wedge m_\partial  z) } \\
= \omega_C^\varphi \big(    \sum (-1)^{|Q|} \epsilon_Q \cdot  \omega_{\partial_1} (\mathsf{A}(T_1) \wedge  \gamma^1_{Q_1} )  \cdots  \omega_{\partial_p} (\mathsf{A}(T_p) \wedge \gamma^p_{Q_p} )\cdot m_{\bar{Q}_1} \gamma_{\bar{Q}_1}^1 \wedge \cdots \wedge m_{\bar{Q}_p} \gamma_{\bar{Q}_p}^p  \wedge m_\partial  z \big).
\end{dmath*}
Note that for all $i=1,\dots,p$ and~$x \in \largewedge^{n_i} H_{\partial_i}$, we have
\begin{dmath*}
 {x = \sum_{|Q_i|=n_i} \epsilon_{\bar{Q}} \omega_{\partial_i} (x \wedge \gamma^i_{Q_i}) \cdot \gamma^i_{\bar{Q}_i} }= { \sum_{|Q_i|=n_i}  (-1)^{|Q_i|}  \epsilon_{Q_i} \omega_{\partial_i} (x \wedge \gamma^i_{Q_i}) \cdot \gamma^i_{\bar{Q}_i}.}
\end{dmath*}
Moreover, $(-1)^{|Q|}= (-1)^{|Q_1|} \cdots (-1)^{|Q_p|}$ and $\epsilon_Q \epsilon_{Q_1} \cdots \epsilon_{Q_p}$ do not depend on $Q_1,\dots,Q_p$ but only on $n_1=|Q_1|,\dots,n_p=|Q_p|$.
Hence, up to a sign
\begin{dmath*}
\mathcal{A}_T^\psi \big(m_\partial   z \big)   = \omega_C^\varphi(m_1 \mathsf{A}(T_1) \wedge \cdots \wedge m_p \mathsf{A}(T_p) \wedge m_\partial z) 
 =\omega_C^\varphi \big( m( \mathsf{A}_C (\Upsilon(T_1) \otimes \cdots \otimes \mathsf{A}(T_p)) \wedge m_\partial z) 
= \omega_\partial \big( \Upsilon_C (\mathsf{A}(T_1) \otimes \cdots \otimes \mathsf{A}(T_p)) \wedge z).
\end{dmath*}

\end{proof}

 \section{A diagrammatic description}
\label{diagram}

\emph{Broken surfaces} are locally flat immersions in the $3$-ball $B^3$ of disjoint annuli and tori, whose singularities consist of a finite number of circles. Any ribbon tangle in $B^4$ can be projected onto a broken surface in a suitable sense. These projections can be viewed as a way to represent ribbon tangles similarly to diagrams in usual knot theory. Conversely, every broken surface is the projection of a ribbon tangle. This gives a  correspondence between ribbon tangles and broken surfaces (in fact, of a certain type, called \emph{symmetric}). 

\emph{Welded} diagrams are a quotient of virtual diagrams under a certain set of moves, see for example \cite{Au}.
To each welded tangle one may associate a symmetric  broken surface diagram, see \cite{S} and~\cite{YAJ2}. This define a map $Tube$ sending  any welded tangle to the ribbon tangle associated to the symmetric broken surface resulting from the preceding construction.

For general ribbon knotted objects, such as ribbon tangles, ribbon tubes~\cite{ABMW} and knotted spheres~\cite{YAN}, the Tube map is well defined and surjective, but its injectivity is still an open question~\cite{WKO1}. However, this is an isomorphisms on ribbon braids and extended ribbon braids~\cite{BH, Dam}.

A \emph{combinatorial} fundamental group can be defined for welded tangle diagrams, with the Wirtinger method (virtual crossing are simply ignored). 
This group coincides with the fundamental group of the complement of an associated ribbon tangle in~$B^4$~\cite{S,Yaj, ABMW}. Then, two welded tangle representing the same  ribbon tangle have isomorphic fundamental groups (and this isomorphism sends meridian to meridian).

Let $D$ be the unit disk in $\mathbb{C}$, and for any positive integer $n$, let $x_1,\dots,x_{2n}$ be a fixed ordered set of points in $\partial D$.
\begin{de}
Let $n$ be a positive integer.
A \emph{welded tangle} on $n$-strands, or \emph{welded $n$-tangle} is a proper embedding $\tau$ of an oriented $1$-manifold  
 in $D$. It consists of some copies of the circle and $n$ copies of the unit interval whose boundary are $\{x_1,\dots,x_{2n} \}$.  The singular set of $\tau$ is a finite number
of transversal double points  equipped with a partial order on the preimages.
By convention, this order is specified by erasing a small neighbourhood of the lowest preimage, or by
pointing the crossing if the preimages are not comparable.
If the preimages of a crossing are comparable, then the crossing is said \emph{classical}; otherwise it is said virtual.
Moreover, a classical crossing is \emph{positive} if the basis made of the tangent vectors of the highest and
lowest preimages is positive; otherwise, it is \emph{negative}.
\end{de}

\begin{figure}[htb] 
	\centering
		\includegraphics[scale=0.5]{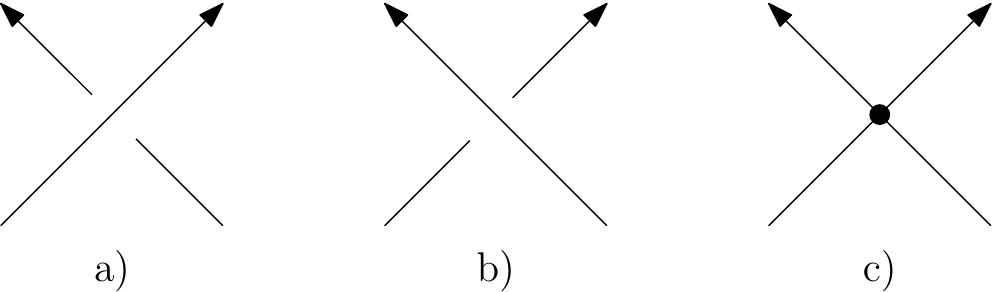} 
	\caption{Positive, negative and virtual crossings.}
	\label{F:Crossings}
\end{figure}

Let $\mu$ be a positive integer.
A $\mu$-\emph{colored} welded tangle is a pair $(\tau,\psi)$ where $\tau$ is a welded tangle and $\psi$ is a map from the set of strands and circles to the set~$\{ t_1,\dots, t_\mu \}$.  
Two welded colored tangles are equivalent if they are related by \emph{generalized Reidemeister moves} (see~\cite[pp. 445 and 454]{KAM}), respecting the coloring.

\subsection{Computation of \texorpdfstring{$\mathsf{A}$}{}}
\label{SS:AlexMatrix}

Let $(\tau,\psi)$ be a $\mu$-colored welded tangle. It decomposes into a finite union of disjoint oriented arcs. Label the crossings with (formal) letters, and each arc with the same letter as the crossing it begins at. If an arc connects points on the border of $\tau$ without meeting any crossing, we use the convention of Figure~\ref{F:matrix2}.
We construct a matrix $M^\psi(\tau)$ with coefficients in $G$ where the rows are indexed by crossings (positive, negative and virtual) and points interrupting arcs, and the columns by the arcs.

\begin{itemize}
\item Fill row corresponding to each positive and negative crossing as shown in Figure~\ref{F:matrix}, 

\begin{figure}[htb]
	\centering
		\includegraphics[scale=0.7]{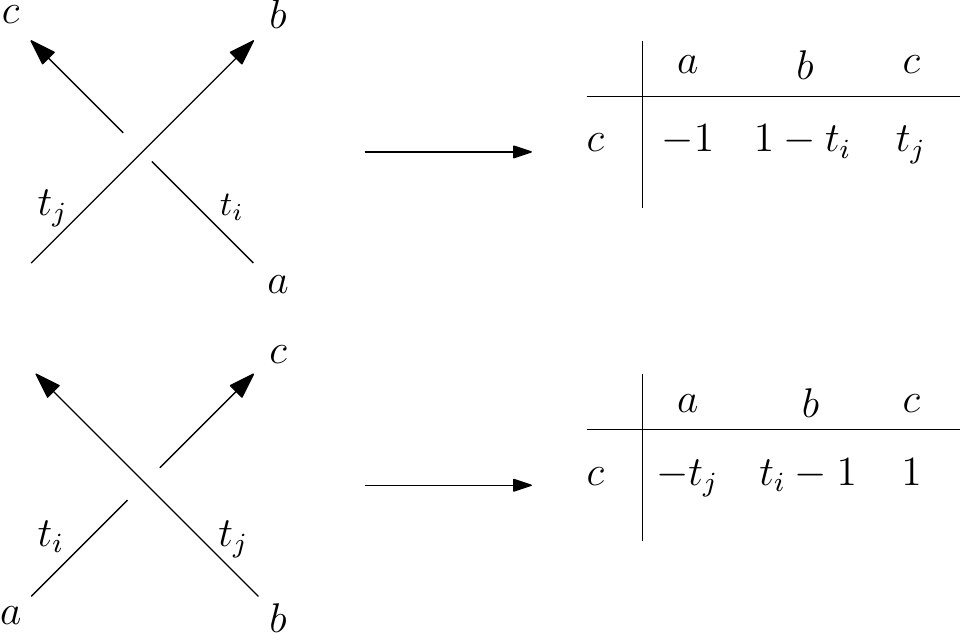}
	\caption{The rule to fill the matrix $M^\psi(T)$, where $t_i$ and $t_j$ are not necessarily different. If $b=a$ or $b=c$ we add the contributions.}
	\label{F:matrix}

\end{figure}

\item At each point on the diagram, fill the row as shown in Figure~\ref{F:matrix2}.

\begin{figure}[htb]
	\centering
		\includegraphics[scale=0.7]{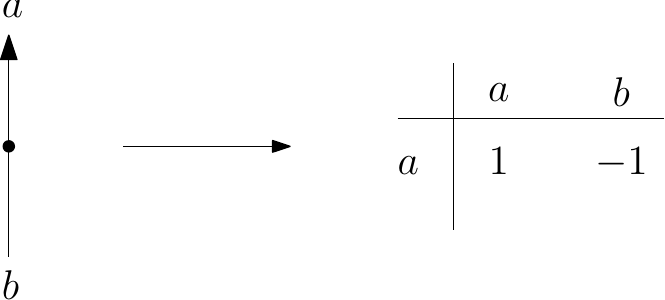}
	\caption{Rule for arcs that don't begin at crossings.}
	\label{F:matrix2}
\end{figure}
\end{itemize}
The other entries of the rows are zero. Virtual crossings can be ignored or considered as divided arcs. Notice that, after some Reidemeister moves of type I, one might suppose that there every arc begins at a crossing, and the receipt of Figure 3 becomes useless to construct the matrix $M^\psi(\tau)$.

\begin{rmk} \label{size}
Let $p$ be the number of internal arcs of~$\tau$. Since $\tau$ has $2n$ arcs connected to the boundary, the total number of arcs is~$2n+p$. One easily observes that the matrix $M^\psi(\tau)$ has size~$(p+n) \times (p +2n)$.
\end{rmk}

\begin{de}
Let $(\tau,\psi)$ be a $\mu$-colored welded $n$-tangle  and $H_\partial$ be the module of rank $2n$ freely generated by the set of marked points~$\{ x_1,\dots,x_{2n} \}$ . The invariant $\alpha$ is defined to be
\[ \alpha(\tau,\psi) = \sum_{I} \lvert M^\psi(\tau)_{I}  \rvert \cdot x_{I} \in \largewedge^n H_\partial, \]
where the sum is taken for all subset $I \subset \{1,\dots,2n \}$ of $n$ elements,  $\lvert M^\psi(\tau)_{I}  \rvert$ is the determinant of the $(n+p)$-minor of $M^\psi(\tau)$ corresponding to the columns indexed by the internal arcs and the columns relative to the arcs indexed by~$I$, and $x_{I}$ is the wedge product of the generators $x_i$ with~$i\in I$.
\end{de}

A computation shows that $\alpha(\tau,\psi)$ is invariant by generalized Reidemeister moves, up to multiplication by a unit. Otherwise the invariance is simply a consequence of Theorem~\ref{weld} below.

\begin{ex} \label{ex2}
Consider the tangle $\tau$ given by one positive crossing, see Figure~\ref{F:Crossings}. The matrix  $M^\psi(\tau)$ coincides with the matrix of Example~\ref{ex1}. The module $H_\partial$ is generated by $x_1,\dots,x_4$ and
\[ \alpha(\tau,\psi)=  t_2 x_3 \wedge x_4 + x_2 \wedge x_3 - t_2 x_1 \wedge x_4 + (t_1 -1) x_1 \wedge x_3 + x_1 \wedge x_2 \in \largewedge^2 H_\partial.\]
\end{ex}

Let $\tau$ be a welded tangle. We denote $\pi(\tau)$ the group defined by the Wirtinger method (ignoring the virtual crossings). Then, there is a system of generators of $\pi(\tau)$ in one-to-one correspondance with the arcs of~$\tau$.  In particular, a $\mu$-coloring of a welded tangle $\tau$ can be viewed as a group homomorphism $\psi$ from $\pi(\tau)$ to the abelian group freely generated by~$t_1,\dots,t_\mu$. 

The following proposition follows directly from the results of Satoh and Yajima \cite{S,Yaj}.

\begin{propo} \label{wirt} 
Let $\tau$ be a welded tangle. For any ribbon tangle $T$ such that $T$ is the image of $\tau$ by the Tube map, there is an isomophism \[ \pi(\tau) \simeq \pi_1(B^4 \setminus T)\]
 sending arcs of $\tau$ to meridians of~$T$.
\end{propo}

\begin{thm} \label{weld}
Let $\mu$ be a positive integer and $G$ be a free abelian group of rank~$\mu$. 
Let $(\tau,\psi)$ be a $\mu$-welded tangle and $(T,\varphi)$ be a $G$-colored ribbon tangle, such that $T$ is the image of $\tau$ by the Tube map. Suppose that there are generators $t_1,\dots, t_\mu$ of $G$ such that the following diagram commutes:

\[
\begin{CD}
\pi_1(B^4\setminus T) @>{\varphi} >> G \\
@V{\simeq}VV @VV{\simeq}V \\
\pi(\tau) @>{\psi}>> \langle t_1,\dots,t_\mu \rangle
\end{CD}
\]

 Then, 
\[  \mathsf{A}(T,\varphi) =\alpha (\tau,\psi)  \in  \largewedge^n H_\partial. \] 
\end{thm}

\emph{Welded string links} are string links with possibly virtual crossings;  {for details on these last ones see~\cite{KLW}}. Through the Tube map, ribbon tubes can be described by welded string links, see~\cite[Section 3.3]{ABMW}. By Theorem \ref{ind}, Proposition~\ref{burau} and Theorem \ref{weld}, the invariant $\alpha$ induces the usual  (generalisation of) colored Burau representation - or Gassner, if the coloring is maximal - on (welded) string links~\cite{BA}.

\begin{proof}
The points of $\tau \cap \partial B^2$  are in one-to-one correspondence with the component of the trivial link $L$ in $T \cap \partial B^4$, and $H_\partial$ is generated by~$x_1,\dots, x_{2n}$. By Proposition \ref{wirt} and Fox calculus, 
the matrix $M^\psi(\tau)$ is a presentation matrix of the $R$-module~$H_1^\varphi(X_T,\ast)$,
 viewed as a $\mathbb{Z}[t_1^{\pm 1},\dots,t_\mu^{\pm 1}]$-module through the choice of generators of~$G$. Then  $M^\psi(\tau)$ is used to compute~$\mathcal{A}^\varphi$. By definition, for all  $I \subset \{1,\dots,2n \}$ with cardinal~$n$:
\[ \omega_\partial (\mathsf{A}(T,\varphi) \wedge x_I)= \mathcal{A}^\varphi(m_\partial x_I).\]
 To calculate
 $\mathcal{A}^\varphi(m_\partial x_I)$, we consider the matrix~$M^\psi(\tau)$, add $n$ row vectors giving the element $m x_{i_1}, \dots, m x_{i_n}$ and  compute the determinant of the resulting square matrix. Hence adding $m x_{i_j}$ corresponds to add the $p+n+j^{\text{th}}$ row $(0,\dots,0,1,0,\dots,0)$ where $1$ is at position~$p+j$. 
We obtain $\mathcal{A}^\varphi(m_\partial x_I) = \epsilon_{\bar{I}} \lvert M^\psi(\tau)_{\bar{I}} \rvert$, where $\bar{I}$ is the complement of $I$ and $\epsilon_{\bar{I}}$ is the signature of the permutation~{$I \bar{I}$} (where the elements of $I$ in increasing
order are followed by the elements of $\bar{I}$ in increasing order). 
Finally, we get \[\mathsf{A}(T,\varphi)= \sum_I \epsilon_{\bar{I}} \cdot \omega_\partial(\mathsf{A}(T,\varphi) \wedge x_I) \cdot x_{\bar{I}} =\sum_I  \lvert M^\psi(\tau)_{\bar{I}}  \rvert \cdot x_{\bar{I}} = \alpha(\tau,\psi).\]
\end{proof}

\begin{rmk}
The invariant $\alpha$ coincides up to a unit with the invariant of virtual tangles introduced by Archibald~\cite{Arc}. It is worth mentioning that using a specific canonical choice for the marked points~$x_i$, her construction is well-defined, not only up to multiplication by a unit (through the multiplication by a correction term).   
\end{rmk}

\begin{proof} of Proposition~\ref{pres}.
As observed in the proof of Theorem~\ref{weld}, the matrix $M^\psi(\tau)$ is a presentation matrix of 
 $H_1^\varphi(X_T,\ast)$, whose size is $(p+n) \times (p+2n)$, see Remark~\ref{size}. We get a presentation of deficiency~$n$.
\end{proof}

\subsection{The circuit algebra \texorpdfstring{$\mathsf{Weld}_\mu$}{}}


\begin{de}
\label{D:circuitDiagrams}
A \emph{circuit $p$-diagram} consists of the following data:
\begin{itemize}
	\item the unit disk $D=D_0$ in $\mathbb{C}$ 
  together with a finite set of disjoint subdisks $D_1, \dots, D_p$ in the interior of~$D$. 
	For every $i$ in~$\{0, \dots, p\}$, each $D_i$ have $2n_i$ distinct marked points  with sign on  its boundary (with $n=n_0$), and a base point $\ast$ on the boundary of each disk.
	\item a finite set of embedded oriented arcs whose boundary are marked points in the~$D_i$. They may cross each other along virtual crossings \emph{only}. Each marked point is the boundary point of some string -which meets the corresponding disk transversally- and the sign is coherent with the orientation. 
\end{itemize}

\end{de}

Circuit diagrams encode only the matching of marked points, see Figure~\ref{F:Circuit}.

\begin{figure}[hbt]
\includegraphics[scale=0.1]{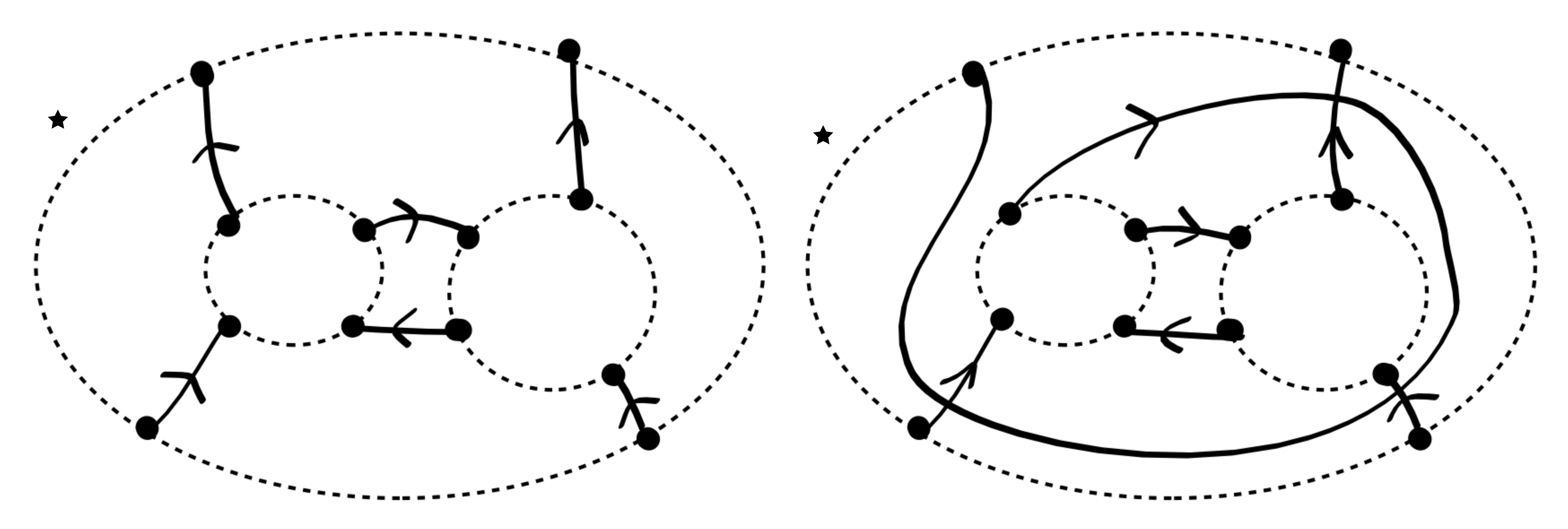}
\caption{Two equivalent circuit diagrams.}
\label{F:Circuit}
\end{figure}

\begin{de}
Consider a circuit $p'$-diagram $P'$ and  a circuit $p''$-diagram $P''$ such that~$D'_i$ is a disk of $P'$  with~$n'_i = n''$, for some $i \in \{1,\ldots,p'\}$. If the signs of the marked points match, we define the diagram $P=P' \circ_i P''$ by rescaling via isotopy the tangle $P''$ so that the boundary of~$D''$ is identified with the boundary of~$D'_i$, and making its marked and base points coincide to those of $D'_i$. Then $D'_i$ is removed to obtain~$P' \circ_i P''$. 
This operation is well defined since the starting points eliminate any rotational ambiguity.
\end{de}

A $\mu$-\emph{colored} circuit diagram is a pair $(P,\psi)$ where $P$ is a circuit diagram and $\psi$ is a map from the set of arcs of $P$ to the set~$\{ t_1,\cdots,t_\mu \}$. 
Two $\mu$-colored circuit diagrams can be composed once the coloring match on the boundary components. The $\mu$-colored circuit diagrams form and operad~$\mathcal{D}_\mu$.
Let $S=\mathbb{Z}[t_1^{\pm 1},\dots,t_\mu^{\pm 1}]$ be the Laurent polynomial ring. 
Let $\mathcal{H}om_\mu$ be the operad of tensor products of $S$-modules and $S$-linear maps. The circuit algebra $\mathsf{Weld}_\mu$ is constructed as a morphism from the operad $\mathcal{D}_\mu$ to~$\mathcal{H}om_\mu$ as follows. 

 Consider the unit circle with a base point and a set of marked points~$X= \{ x_1,\dots,x_{2k} \}$, for~$k \geq 0$ (with a sign).
 To this data, we associate the module~$\largewedge^k H_\partial$, where $H_\partial$ is the free $S$-module of rank $2k$ generated by~$X$.
 Let $(P,\psi)$ be a $\mu$-colored circuit diagram, and $M=\{c_1,\dots, c_q \}$ be the set of curves of~$P$. Consider the free module $H$ generated by~$M$, and the volume form $\omega\colon \largewedge^q H \to S$ related to this basis.  For~$i=1,\dots,p$, denote $H_{\partial_i}$ the module associated to the boundary circle~$\partial D_i$ and $H_\partial$ he module associated to $\partial D$. Let $m_i\colon H_{\partial_i} \to H$ be the morphims defined by $m_i(x_j)= \text{sign}(x_j) c_j$ if $x_j \in \partial c_j$. The mophism $m_\partial$ is defined similarly. Let $\omega_{\partial_i}$ be the volume form on $H_{\partial_i}$ related to the generating system of points of the circle~$\partial D_i$. Set $m = \otimes_{i} \left( \largewedge^{n_i} m_{\partial_i} \right)$.
 To the colored diagram $(P,\psi)$ we associate \[ \gamma_{P,\psi} \colon \bigotimes_{i=1}^p \largewedge^{n_i} H_{\partial_i}
 \to \largewedge^{n} H_\partial \] such that, for~$x \in \otimes_{i} \big( \largewedge^{n_i} H_{\partial_i} \big)$,
\begin{equation*}
 \omega_C^\varphi ( m(x) \wedge m_\partial(y) ) = \omega_\partial ( \gamma_{P,\psi} (x) \wedge y), \ \forall y \in \largewedge^n H_\partial.
\end{equation*}
Then, we can prove Proposition \ref{wel} below, by repeating the arguments of the proof of Theorem \ref{cob}.

\begin{propo} \label{wel}
$\mathsf{Weld}_\mu$ is a circuit algebra. 
\end{propo}

Note that $\mathsf{Weld}_\mu$ is similar to half densities introduced by Archibald, see~\cite{Arc} for the definition of half densities and~\cite{Dam} for the explicit correspondance. The morphism $\gamma_{P,\psi}$ could be written as the interior product relative to a subset corresponding to interior arcs of~$P$.

The Tube map and the choice of a set of generators $\{t_1,\dots,t_\mu\}$ of $G$  induce a surjective morphism of  algebras
 \[ \mathsf{Weld}_\mu \longrightarrow   \mathsf{C}ob_G .\]

Given a colored circuit diagram $P$  and a collection of colored tangle diagrams $\tau_1,\ldots,\tau_p$, one may create a new tangle, if the data on the boundaries and the colorings match, 
 by gluing $\tau_i$ into the internal disk $D_i$ of~$P$.
Similarly to Theorem \ref{morph}, the invariant $\alpha$ commutes with this action of $\mu$-colored circuit diagrams on $\mu$-colored welded tangles.
Indeed, one has the following proposition.

\begin{propo}
Let $(\tau,\psi)$ be the $\mu$-colored welded tangle obtained by gluing the $\mu$-colored welded tangles $(\tau_1, \psi_1), \cdots, (\tau_p,\psi_p)$ to a $\mu$-colored circuit diagram~$(P,\chi)$.
The following equality holds
\[ \alpha(\tau,\psi) = \gamma_{P,\chi} \big( \alpha(\tau_1,\psi_1) \otimes \cdots \otimes \alpha(\tau_p,\psi_p) \big) \in \largewedge^n H_\partial.\]
\end{propo}

\section{Examples}
\label{ex}

Consider the tangle diagram~$\tau$, given in Figure~\ref{tau}. We let $G=\mathbb{Z}= <t>$ and $\psi$ be the coloring sending all arcs of $\tau$ to~$t$.
In this section, we compute $\alpha(\tau,\psi)$ in different ways. By Theorem~\ref{weld}, this computes the value of $\mathsf{A}(T,\varphi)$ the image of $(\tau,\psi)$ by the Tube map. 

\begin{figure}[hbt] 
\includegraphics[scale=0.5]{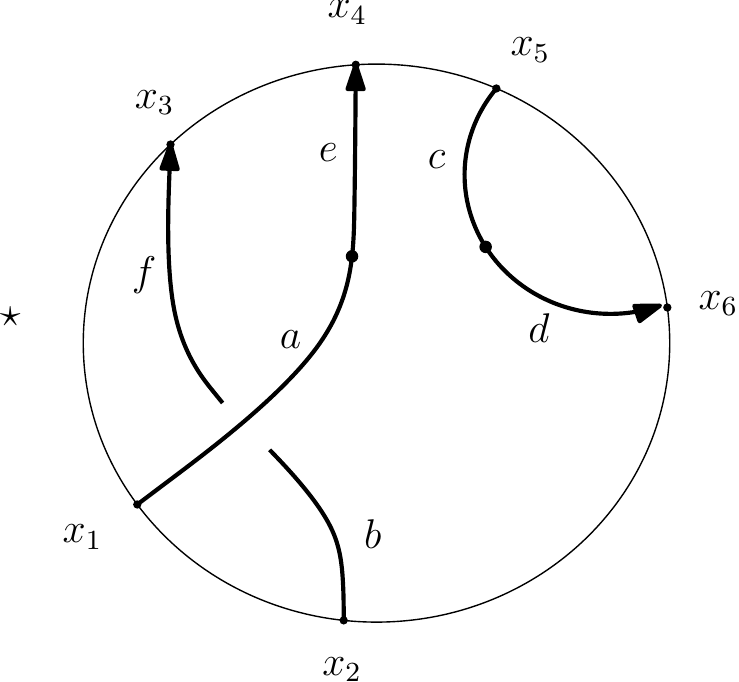}
\caption{A welded tangle $\tau$.}
\label{tau}
\end{figure}

First, we compute $\alpha(\tau,\psi)$ directly. Label the arcs of $\tau$ with letters $a$ to~$f$ as in Figure~\ref{tau}. We obtain the matrix 
\[
M^\psi(\tau)= 
\bordermatrix{
  & a & b & c & d & e & f  \cr
 & 0 & 0 & -1  & 1 & 0 & 0  & \cr
 & -1 & 0 & 0  & 0 & 1 & 0  & \cr
 & 1-t & -1 & 0 & 0 & 0 & t &
}
 \] 
The $\mathbb{Z}[t^{\pm 1}]$-module $H_\partial$ is free, generated by~$x_1,\dots, x_6$ and  $\alpha(\tau,\psi) \in \largewedge^3 H_\partial$ is given by
\begin{dmath*}
\alpha(\tau,\psi) = -x_1 \wedge x_2 \wedge x_5 + x_1 \wedge x_2 \wedge x_6 + x_2 \wedge x_5 \wedge x_4 + (t-1) x_1 \wedge x_5 \wedge x_4 - t x_1 \wedge x_5 \wedge x_3  \\
 + (1-t) x_1 \wedge x_6 \wedge x_4 + t x_1 \wedge x_6 \wedge x_3 + t x_6 \wedge x_4 \wedge x_3 - x_2 \wedge x_6 \wedge x_4 - t x_5 \wedge x_4 \wedge x_3.
\end{dmath*}
We now consider $(\tau,\psi)$ as the composition of the circuit diagram $(P,\psi)$ with $(\sigma,\psi)$, see Figure~\ref{comp}.
\begin{figure}[hbt]
\includegraphics[scale=0.5]{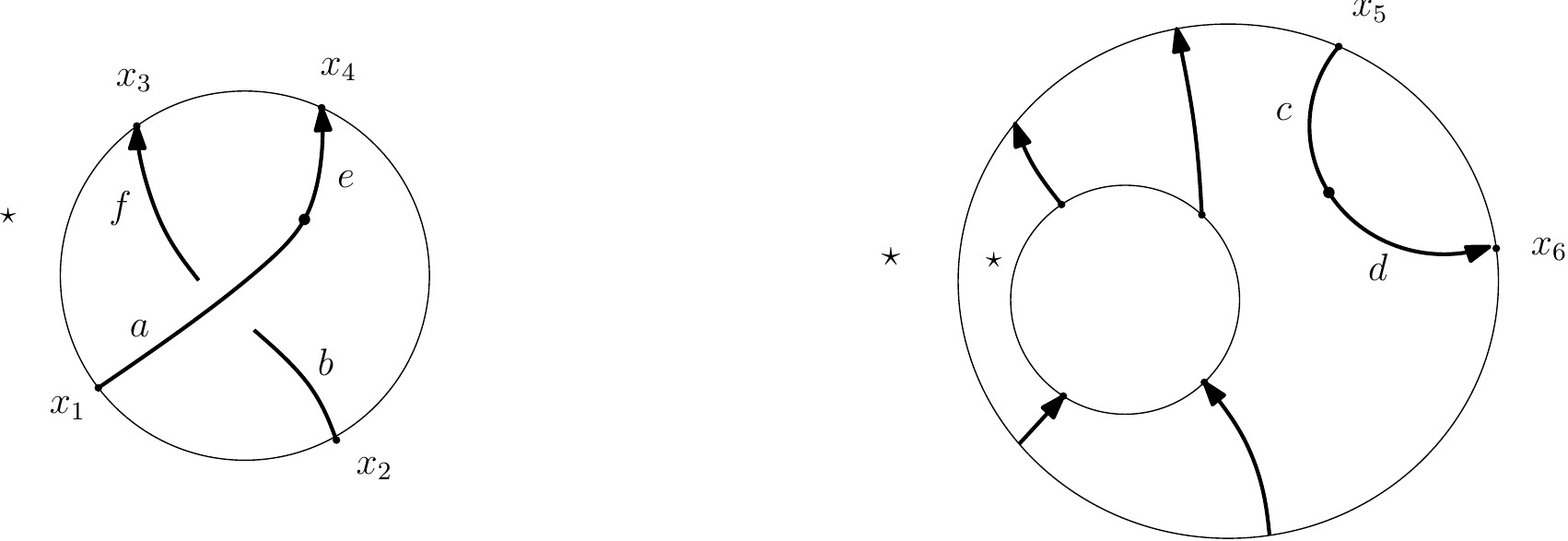}
\caption{A welded tangle $\sigma$ and a circuit diagram $P$.}
\label{comp}
\end{figure}
We have to compute $\gamma_{P,\psi}\colon \largewedge^2 H_{\partial_1} \to \largewedge^3 H_{\partial}$ (here~$p=1$). Let $H$ be the free module generated by the curves of~$\sigma$, labelled~$a,b,e,f$. 
We have that $H_{\partial_1} = \langle x_1,\dots, x_4 \rangle$ and~$H_\partial= \langle x_1,\dots, x_6 \rangle$.  Using the volume form on $H$ related to the choice of the basis~$a,b,c,e,d,f$, and the maps induced by the inclusions $m_1\colon H_{\partial_1} \to H$ and~$m_\partial \colon H_\partial \to H$, we obtain  \[ \gamma_{C,\psi} (x_i \wedge x_j) = x_i \wedge x_j \wedge (x_6 - x_5), \ \forall \ i,j=1,\dots,4,\]
and
\[ M^\psi(\sigma)= 
 \bordermatrix{
  & a & b & e & f \cr
 & -1 & 0 & 1  & 0   & \cr
 & 0 & -1 & 1-t  & t & }
 \] 
We get 
\[
\alpha(\sigma,\psi)=  x_1 \wedge x_2 + (t-1) x_1 \wedge x_4 - t x_1 \wedge x_3 + x_2 \wedge x_4 - t x_3 \wedge x_4.
\]
The composition $\alpha(\tau,\psi)= \gamma_{P,\psi} ( \alpha(\sigma,\psi) )$ gives the result.
Finally, we consider $\tau$ as the composition of the tangle $\sigma \otimes \beta$ with the circuit~$Q$, see Figure~\ref{3}.
\begin{figure}[hbt]
\includegraphics[scale=0.5]{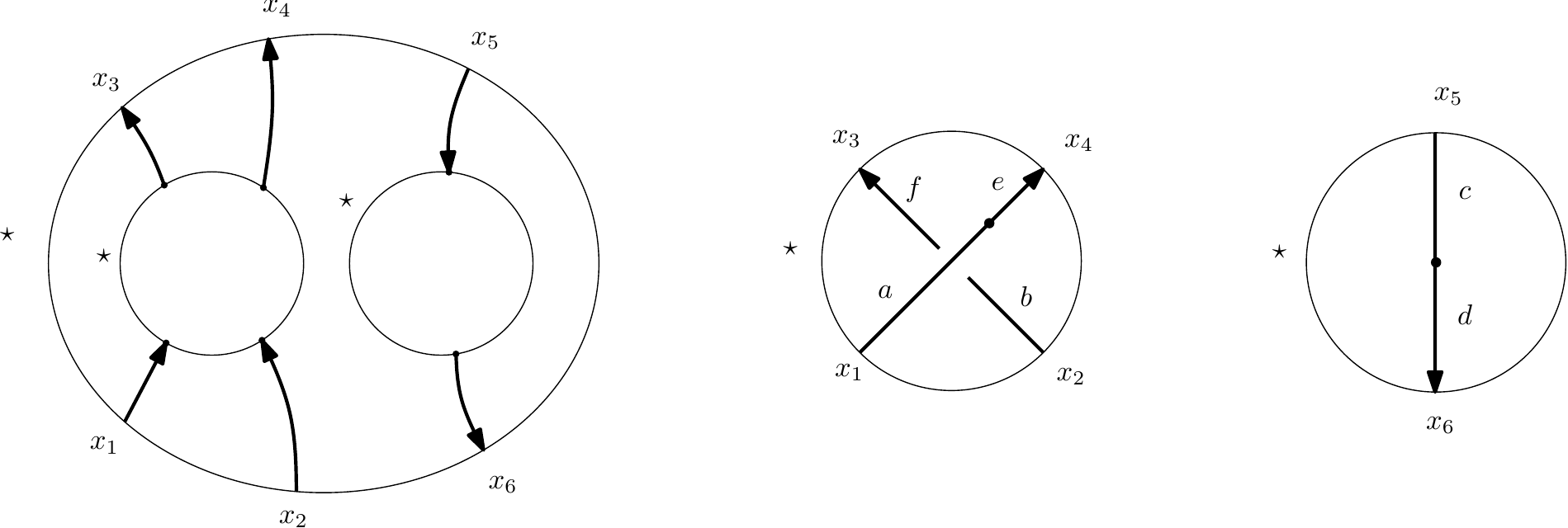}
\caption{A circuit diagram $Q$ representing a cobordism and two welded tangles $\sigma$ and~$\beta$.}
\label{3}
\end{figure}
Here~$p=2$, and \[\gamma_{Q,\psi} \colon \largewedge^2 H_{\partial_1} \otimes H_{\partial_2} \to  \largewedge^3 H_\partial.\] As previously $\alpha(\sigma,\psi)= x_1 \wedge x_2 + (t-1) x_1 \wedge x_4 - t x_1 \wedge x_3 + x_2 \wedge x_4 - t x_3 \wedge x_4 \in \largewedge^2 H_{\partial_1}$ and $\gamma(\beta)=x_6 - x_5 \in H_{\partial_2}$. The composition $\alpha(\tau,\psi)=\gamma_{Q,\psi} \left( \alpha(\sigma,\psi) \otimes \alpha(\beta,\psi) \right)$ gives the result again.

\begin{rmk} \label{bur}
In the sense of Subsection~\ref{SS:Burau}, there is a splitting of the tangle $\sigma$  to an (oriented) braid $\sigma_1$ in~$B_2$. We have
\[
M_0 = \langle x_1, x_2 \rangle \hbox{ and } M_1= \langle x_3, x_4 \rangle,
\]
 and $\alpha(\sigma,\psi) \in \largewedge^2 H_{\partial_1} \simeq \largewedge^2 (M_0 \oplus M_1) \simeq \bigoplus_{k=0}^2 \largewedge^k M_0  \otimes \largewedge^{2-k} M_1$, similarly to the proof of Theorem~\ref{ind}. Then $\alpha(\sigma,\psi)$ decomposes as:
\[\big(-t x_3 \wedge x_4 \big) \oplus \big((t-1) x_1 \otimes x_4 - t x_1 \otimes x_3 + x_2 \otimes x_4 \big) \oplus \big(x_1 \wedge x_2\big).
\]
Let $\omega_0: \Lambda^2 M_0 \rightarrow R$ be the volume form related to the basis $\langle x_1,x_2 \rangle$.
For $k=1$, the element $(t-1) x_1 \otimes x_4 - t x_1 \otimes x_3 + x_2 \otimes x_4 \in M_0  \otimes M_1$ induces the morphism $M_0 \rightarrow M_1$:
$$
x \mapsto (t-1) \omega_0(x \wedge x_1) \cdot x_4 - t \, \omega_0(x \wedge x_1) \cdot x_3 + \omega_0(x \wedge x_2) \cdot x_4.  
$$
The image of $x_1$ is $x_4$ and the image of $x_2$ is $(1-t)x_4 + t x_3$. This corresponds to the Burau representation. 
 Similalrly, the other values of $k$ give $k^{\text{th}}$-exterior powers of Burau, up to a sign.
\end{rmk}
\section*{Acknowledgements}
The authors thank B. Audoux for accurate comments on the manuscript. They thank P.~Bellingeri, J.-B.~Meilhan and E.~Wagner for helpful comments, and G.~Massuyeau who brought to our attention the similarities between Archibald's invariant and the functor constructed by Bigelow, Cattabriga and the second author.
\end{proof}



\bibliographystyle{amsalpha}
\bibliography{AlexWelded4}

\end{document}